\documentclass[reqno]{amsart}
\usepackage{amsmath,amsfonts,amssymb,amscd,verbatim,multicol}
\usepackage{mathrsfs}  
\usepackage{fullpage}
\usepackage{xcolor}
\allowdisplaybreaks

\usepackage{enumitem}
\usepackage{hyperref}

\numberwithin{equation}{section}

\theoremstyle{definition}
\newtheorem{theorem}{Theorem}[section]

\newtheorem{corollary}[theorem]{Corollary} 
\newtheorem{definition}[theorem]{Definition} 

\newtheorem{lemma}[theorem]{Lemma}
\newtheorem{proposition}[theorem]{Proposition}

\newtheorem{thm}{Theorem}

\newtheorem*{definition*}{Definition} 
\newtheorem*{example*}{Example}

\DeclareMathOperator\ord{ord}

\renewcommand\int{\mathrm{int}}

\newcommand\cnt{\mathcal Z}
\newcommand\pcnt{\cnt_P}
\newcommand\inv{^{-1}}

\newcommand\iso{\cong}
\newcommand\kk{\Bbbk}

\newcommand\cX{\mathcal X}

\newcommand\ZZ{\mathbb Z}

\newcommand\ba{\mathbf a}
\newcommand\bb{\mathbf b}

\newcommand\be{\mathbf e}
\newcommand\bp{\mathbf p}
\newcommand\bq{\mathbf q}

\newcommand\bfs{\mathbf s}
\newcommand\bt{\mathbf t}

\newcommand\bzero{\mathbf 0}

\newcommand\balpha{\boldsymbol \alpha}
\newcommand\bbeta{\boldsymbol \beta}
\newcommand\bpartial{\boldsymbol \partial}

\newcommand\bnabla{\boldsymbol \nabla}
\newcommand\btheta{\boldsymbol \theta}

\newcommand\fsl{\mathfrak{sl}}

\newcommand\grp[1]{\langle #1 \rangle}
\newcommand\restrict[1]{\raisebox{-.3ex}{$|$}_{#1}}

\begin{document}

\title{Twisted generalized Weyl Poisson algebras of type $(A_1)^n$}

\author[Gaddis]{Jason Gaddis}
\address{Miami University, Department of Mathematics, Oxford, Ohio 45056} 
\email{gaddisj@miamioh.edu}


\begin{abstract}
We introduce a generalization of generalized Weyl Poisson algebras. This is a Poisson analogue of the twisted generalized Weyl algebras defined by Mazorchuk and Turowska. We prove existence of these algebras in two ways, using Ore extensions and by using skew Laurent Poisson algebras. It is shown that this structure is preserved under tensor products, Poisson twists, and by taking invariant rings. Finally, we prove a simplicity criterion for these Poisson algebras.
\end{abstract}

\maketitle

\section{Introduction}

Throughout, let $\kk$ be a field. All algebras are $\kk$-algebras. A \emph{Poisson algebra} is a pair $(R,\{,\})$ such that $R$ is a commutative algebra, $R$ is a Lie algebra under $\{,\}$, and $\{r,-\}:R \to R$ is a derivation for each $r \in R$. For $n \in \ZZ_+$, $[n] = \{1,\hdots,n\}$.

Generalized Weyl algebras (GWAs) include the classical Weyl algebras, primitive quotients of $U(\fsl_2)$, as well as various families of quantum algebras. Mazorchuk and Turowska introduced a generalization of this, called \emph{twisted generalized Weyl algebras} (TGWAs) \cite{MT}. Within this class of algebras, Hartwig identified those of Cartan type \cite{hart3}. In this work, we are particularly interested in those of Cartan type $(A_1)^n$, which include various versions of multiparameter quantized Weyl algebras \cite{FH1}.

A Poisson analogue of the GWA construction was first observed by Cho and Oh, arising from semiclassical limits of quantum generalized Weyl algebras \cite{CO}. Bavula generalized this construction to higher rank algebras called generalized Weyl Poisson algebras (GWPAs) \cite{Bpoisson}. These also appear in the work of Castellan \cite{castellan}. In this work, we introduce an analogue of TGWAs for Poisson algebras.

\begin{definition}\label{defn.tgwpa}
Fix $n \in \ZZ_+$. Let $R$ be a Poisson algebra, let $\bpartial = (\partial_1,\hdots,\partial_n)$ be an $n$-tuple of commuting Poisson derivations of $R$, and let $\ba=(a_1,\hdots,a_n)$ be an $n$-tuple of Poisson central regular elements of $R$. Assume further that there exist $\mu,\gamma \in M_n(\kk)$ such that $\partial_j(a_i) = \gamma_{ij} a_i$ for $i \neq j$ and 
\begin{align}\label{eq.gamma}
	\gamma_{ij}+\gamma_{ji} = \mu_{ij} + \mu_{ji} \qquad \text{for all $i,j \in [n], i \neq j$.}
\end{align}
Given the above data, a \emph{twisted generalized Weyl Poisson algebra of type $(A_1)^n$} (of rank $n$) is
\[ R[x_1,y_1,\hdots,x_n,y_n]/(y_1x_1-a_1,\hdots,y_nx_n-a_n)\]
with Poisson bracket
\begin{align*}
&\{r,x_i\} = \partial_i(r)x_i, \quad 
	\{ r,y_i \} = -\partial_i(r)y_i,  \quad
	\{y_i,x_i\} = \partial_i(a_i),
	& &i \in [n], r \in R, \\
&\{y_i,x_j\} = \mu_{ij}x_jy_i, \quad 
\{x_i,x_j\} = (\gamma_{ij}-\mu_{ij}) x_i x_j, \quad
\{y_i,y_j\} = (\mu_{ji}-\gamma_{ij}) y_i y_j  & &i,j \in [n], i \neq j.
\end{align*}
We denote the above by $A_\mu(R,\bpartial,\ba)$.
\end{definition}

The use of \emph{type $(A_1)^n$} in Definition \ref{defn.tgwpa} is meant in analogy to the type $(A_1)^n$ TGWAs discussed in \cite{FH1,hart3}. Because we do not consider other types in this work, we will use TGWPA to refer to the algebras in Definition \ref{defn.tgwpa}. 

When $\gamma_{ij}=\mu_{ij}=0$ for all $i,j \in [n]$, we obtain the GWPAs in \cite{Bpoisson}. If, additionally, $R=\kk[h_1,\hdots,h_n]$, $\ba = (h_1,\hdots,h_n)$, and $\bpartial=\left( \frac{d}{dh_1}, \hdots, \frac{d}{dh_n} \right)$, then $A_\mu(R,\bpartial,\ba)$ is the $n$th Weyl Poisson algebra. 

In analogy to the TGWA case, we call \eqref{eq.gamma} the \emph{consistency equations}. These are implied by the antisymmetry of the bracket. In particular, for $i,j \in [n]$ and $i \neq j$, we have
\[ (\gamma_{ij}-\mu_{ij}) x_i x_j = \{x_i,x_j\} = -\{x_j,x_i\} = -(\gamma_{ji}-\mu_{ji}) x_j x_i,\]
so $\gamma_{ij}-\mu_{ij} = -(\gamma_{ji}-\mu_{ji})$, which is clearly equivalent to \eqref{eq.gamma}.

Given a TGWPA $A$ of rank $n$, there is a $\ZZ^n$-grading obtained by setting $\deg(x_i)=\be_i$, $\deg(y_i)=-\be_i$, and $\deg(r)=0$ for all $r \in R$, where $\be_i \in \ZZ^n$ is the $i$th standard basis vector. This $\ZZ^n$-grading respects the Poisson bracket, so we say that $A$ is a \emph{$\ZZ^n$-graded Poisson algebra}.

In Section \ref{sec.exist}, we establish existence of TGWPAs in two different ways.

\begin{thm}
Let $A=A_\mu(R,\bpartial,\ba)$ be a TGWPA.
\begin{itemize}
\item (Theorem \ref{thm.construct}) $A$ is a quotient of an iterated Poisson Ore extension.
\item (Theorem \ref{thm.rees}) $A$ is a subalgebra of a skew Laurent Poisson algebra.
\end{itemize}
\end{thm}

In Section \ref{sec.construct}, we study how the TGWPA structure is preserved under several standard constructions. We also show that the multiparameter quantized Weyl Poisson algebras may be realized as TGWPAs.

\begin{thm}
Let $A$ and $A'$ be TGWPAs.
\begin{itemize}
\item (Proposition \ref{prop.tensor}) The tensor product $A \otimes A'$ is a TGWPA.
\item (Theorem \ref{thm.twist}) Let $\delta$ be a twisting system as defined in Section \ref{sec.twist}. Then the Poisson twist $A^\delta$ is a TGWPA.
\item (Theorem \ref{thm.invariants}) Let $\phi$ be a diagonal automorphism of $A$ as defined in Section \ref{sec.invariants}. The invariant ring $A^{\langle \phi \rangle}$ is again a TGWPA.
\end{itemize}
\end{thm}

Finally, in Section \ref{sec.simple}, we establish a simplicity criterion for TGWPAs. This should be compared to \cite[Theorem 1.1]{Bpoisson} in the GWPA case and \cite[Theorem 7.20]{HO} in the TGWA case. In this result we restrict the characteristic of the field $\kk$.

\begin{thm}[Theorem \ref{thm.simple}]
Suppose $\mathrm{char}{\kk}=0$.
Let $A=A_\mu(R,\bpartial,\ba)$ be a TGWPA of rank $n$ over a Poisson algebra $R$. Then $A$ is Poisson simple if and only if the following hold:
\begin{enumerate}
\item $R$ has no nonzero proper $\bpartial$-invariant Poisson ideals,
\item there is no $\balpha \in \ZZ^n\backslash\{\bzero\}$ and nonzero $u \in R$ such that $u\partial_{\balpha}(r)=\{u,r\}$ for $r \in R$ and $\partial_i(u)=u \sum_j \alpha_j (\gamma_{ij}-\mu_{ij})$ for all $i$, and
\item $Ra_i + R\partial_i(a_i)=R$ for all $i=1,\hdots,n$.
\end{enumerate}
\end{thm}

\subsection*{Acknowledgements}
The author is partially supported by an AMS–Simons Research Enhancement Grant for PUI Faculty.

\section{Existence of TGWPAs}
\label{sec.exist}

In this section, we show that TGWPAs appear through multiple constructions. First, we provide background on Poisson Ore extensions and show that TGWPAs are quotients of certain iterated Poisson Ore extensions. Secondly, we discuss skew Poisson Laurent algebras and show that TGWPAs are subalgebras of these. 

We now set up some notation that will be used throughout this section. Fix $n \in \ZZ_+$. Let $R$ be a Poisson algebra, let $\bpartial = (\partial_1,\hdots,\partial_n)$ be an $n$-tuple of commuting Poisson derivations of $R$, and let $\ba=(a_1,\hdots,a_n)$ be an $n$-tuple of regular Poisson central elements of $R$. Assume further that there exist $\mu_{ij},\gamma_{ij} \in \kk$ such that $\partial_j(a_i) = \gamma_{ij} a_i$ and $\gamma_{ij}+\gamma_{ji} = \mu_{ij} + \mu_{ji}$ for all $i \neq j$. 

\subsection{Poisson Ore extensions}

A derivation $\partial$ of the Poisson algebra $R$ is a \emph{Poisson derivation} if 
\begin{align}\label{eq.deriv} 
	\partial(\{r,s\}) = \{\partial(r),s\} + \{r,\partial(s)\} \quad\text{for all $r,s \in R$.}
\end{align}
A Poisson derivation $\partial$ of $R$ is said to be \emph{inner} if there exists nonzero $u \in R$ such that $u\partial(r) = \{u,r\}$ for all $r \in R$. We say a Poisson ideal $I$ of $R$ is \emph{$\partial$-invariant} if $\partial(I)\subset I$. An element $a \in R$ is \emph{Poisson normal} if $\{a,r\} \in aR$ for all $r \in R$.

The \emph{Poisson center} of $R$ is $\pcnt(R)=\{ r \in R \mid \{r,-\}=0 \}$. The following is well-known.

\begin{lemma}\label{lem.cnt}
Let $\partial$ be a Poisson derivation of a Poisson algebra $R$ and let $Z$ be the Poisson center of $R$. Then $\partial(z) \in Z$ for all $z \in Z$.
\end{lemma}
\begin{proof}
Let $r \in R$ and $z \in Z$. Then $0=\{r,z\}$ and so
$0 = \partial(\{r,z\}) = \{\partial(r),z\} + \{r,\partial(z)\} = \{r,\partial(z)\}$.
\end{proof}

Given a Poisson derivation $\partial$ of the Poisson algebra $R$, a derivation $\delta$ of $R$ is called a \emph{Poisson $\partial$-derivation} of $R$ if
\begin{align}\label{eq.pderiv} 
	\delta(\{r,s\}_R) = \{ \delta(r),s \}_R + \{ r, \delta(s) \}_R - \partial(r)\delta(s) + \delta(r)\partial(s)
\quad\text{for all $r,s \in R$.}
\end{align}	
When $\partial = 0$, a Poisson $\partial$-derivation is just a (Poisson) derivation of $R$. The \emph{Poisson Ore extension} associated to $(R,\partial,\delta)$ is the polynomial extension $R[t]$ with bracket
\[
    \{r,s\} = \{r,s\}_R \qquad
    \{r,t\} = \partial(r)t + \delta(r)
    \quad\text{for all $r,s \in R$.}
\]
We denote this by $R[t;\partial,\delta]_P$.

Suppose a Poisson algebra $R$ is generated by a set $X$. To verify that a derivation $\partial$ (resp. $\delta$) of $R$ is a Poisson derivation (resp. Poisson $\partial$-derivation), it suffices to check that it satisfies \eqref{eq.deriv} (resp. \eqref{eq.pderiv}) for all $r,s \in X$ \cite[Lemma 1.2]{OH1}.

The following lemma will simplify some of the computations that follow.

\begin{lemma}\label{lem.short}
Let $S$ be a Poisson algebra containing $R$ and let $\phi$ be a derivation of $S$ such that $\phi(R) \subset R$. Let $p \in \kk$, let $\tau_1,\tau_2$ be commuting Poisson derivations of $R$,
and let $t_1,t_2 \in S$ satisfying
\[ \{ r,t_1 \} = \tau_1(r)t_1, \quad \{ r,t_2 \} = \tau_2(r)t_2, \quad \{t_1,t_2\} = pt_1t_2 \quad\text{for all $r \in R$.}\]
If $\phi\tau_i=\tau_i\phi$ and $\phi(t_i)=q_it_i$, $q_i \in \kk$, for $i=1,2$, then
\begin{align*}
\{ \phi(r), t_i \} + \{ r, \phi(t_i) \} &= \phi( \tau_i(r)t_i ) = \phi( \{r,t_i\}) \\
\{ \phi(t_1), t_2\} + \{ t_1, \phi(t_2) \} &= \phi(p t_1t_2) = \phi(\{t_1,t_2\})
\end{align*}
for all $r \in R$.
\end{lemma}
\begin{proof}
For $i=1,2$, we have
\[
\{ \phi(r), t_i \} + \{ r, \phi(t_i) \}
	= \tau_i(\phi(r))t_i + q_i\tau_i(r)t_i
	= \phi(\tau_i(r))t_i + \tau_i(r)(q_i t_i)
	= \phi( \tau_i(r)t_i ) = \phi( \{r,t_i\}).\]
Moreover,
\[
\{ \phi(t_1), t_2\} + \{ t_1, \phi(t_2) \}
	= (q_1 + q_2)pt_1t_2
	= p(\phi(t_1)t_2 + t_1\phi(t_2))
	= \phi(p t_1t_2) = \phi(\{t_1,t_2\}).
\]
This proves the claim.
\end{proof}

The following is now immediate from Lemma \ref{lem.short}.

\begin{lemma}\label{lem.partial_ext}
For $k=2,\hdots,n$, $\partial_k$ extends to a Poisson derivation of $R[x_1;\partial_1]_P \cdots [x_{k-1};\partial_{k-1}]_P$ by setting
\[ \partial_k(x_i) = (\mu_{ki}-\gamma_{ki})x_i \quad \text{for all $i<k$}.\]
\end{lemma}

Let $T_0 = R[x_1;\partial_1]_P \cdots [x_n;\partial_n]_P$ where each $\partial_k$, $k=2,\hdots,n$, is extended as in Lemma \ref{lem.partial_ext}. Because the bracket is anti-symmetric, we have
\[ (\mu_{ij}-\gamma_{ij})x_ix_j = \{x_j,x_i\} = - \{x_i,x_j\} = -(\mu_{ji}-\gamma_{ji})x_ix_j,\]
which is satisfied by our hypotheses on the scalars.

\begin{lemma}\label{lem.alpha_beta}
For $k \in [n]$, define derivations $\phi_k$ and $\psi_k$ of $T_0$ by 
$\phi_k(r)=-\partial_k(r)$ and $\psi_k(r)=0$ for all $r \in R$, and
\[
\phi_k(x_i) = \begin{cases}
	0 & \text{ if } i = k \\
	-\mu_{ki} x_i & \text{ if } i \neq k,
\end{cases}\qquad
\psi_k(x_i) = \begin{cases}
	-\partial_i(a_i) & \text{ if } i = k \\
	0 & \text{ if } i \neq k,
\end{cases}
\]
for all $i \in [n]$. Then $\phi_k$ is a Poisson derivation of $T_0$ and $\psi_k$ is a Poisson $\phi_k$-derivation of $T_0$.
\end{lemma}
\begin{proof}
That $\phi_k$ is a Poisson derivation of $T_0$ follows from Lemma \ref{lem.short}.
We check that $\psi_k$ is a Poisson $\phi_k$-derivation. This is obvious for $\psi_k$ applied to $\{x_i,x_j\}$ and $\{r,x_i\}$ with $i,j \neq k$. Assume $k\neq j$, then
\begin{align*}
\{\psi_k(x_k),x_j\} &+ \{ x_k, \psi_k(x_j)\} - \phi_k(x_k)\psi_k(x_j) + \psi_k(x_k)\phi_k(x_j) \\
	&= \partial_j(-\partial_k(a_k))x_j + \mu_{kj}\partial_k(a_k)x_j \\
	&= -\partial_k(\partial_j(a_k))x_j + \mu_{kj}\partial_k(a_k)x_j \\
	&= (-\gamma_{kj}+\mu_{kj}) \partial_k(a_k)x_j \\
	&= (\gamma_{kj}-\mu_{kj})( \psi_k(x_k)x_j + x_k\psi_k(x_j) ) \\
	&= \psi_k( (\gamma_{kj}-\mu_{kj}) x_kx_j)
	= \psi_k(\{x_k,x_j\}) \\
\{\psi_k(r),x_k\} &+ \{ r, \psi_k(x_k)\} - \phi_k(r)\psi_k(x_k) + \psi_k(r)\phi_k(x_k) \\
	&= \{ r, -\partial_k(a_k)\} - \partial_k(r)\partial_k(a_k)
	= \partial_k(r)(-\partial_k(a_k)) \quad\text{(by Lemma \ref{lem.cnt})} \\
	&= \psi_k(\partial_k(r))x_k + \partial_k(r)\psi_k(x_k)
	= \psi_k(\partial_k(r)x_k) = \psi_k( \{r,x_k\} ).
\end{align*}
This proves the claim for the $\psi_k$.
\end{proof}

We now extend the maps from Lemma \ref{lem.alpha_beta}.

\begin{lemma}\label{lem.ab_ext}
For $k=2,\hdots,n$, $\phi_k$ and $\psi_k$ extend to a Poisson derivation and Poisson $\phi_k$-derivation, respectively, of $T_0[y_1;\phi_1,\psi_1]_P\cdots[y_{k-1};\phi_{k-1},\psi_{k-1}]_P$ by setting
\[ \phi_k(y_i) = (\gamma_{ki}-\mu_{ik}) y_i \quad\text{and}\quad \psi_k(y_i)=0 \quad \text{for all $i<k$}.\]
\end{lemma}
\begin{proof}
First we verify the claim for the $\phi_k$. By Lemma \ref{lem.short}, we need only check the Poisson bracket between $y_i$ and $x_i$. Suppose $i \neq k$. Then
\begin{align*}
\{ \phi_k(y_i),x_i\} + \{ y_i, \phi_k(x_i) \}
	&= (\gamma_{ki}-\mu_{ik} - \mu_{ki}) \partial_i(a_i)
	= -\gamma_{ik} \partial_i(a_i) \quad\text{(by \eqref{eq.gamma})} \\
	&= -\partial_i(\partial_k(a_i)) 
	= \phi_k (\partial_i(a_i)) = \phi_k(\{y_i,x_i\}).
\end{align*}
This proves the claim for the $\phi_k$. 

It remains to check the claim for the $\psi_k$, wherein it suffices only to consider those brackets involving $x_k$. Suppose $j<k$. Then
\begin{align*}
\{\psi_k(x_k),y_j\} &+ \{ x_k, \psi_k(y_j)\} - \phi_k(x_k)\psi_k(y_j) + \psi_k(x_k)\phi_k(y_j) \\
	&= \{-\partial_k(a_k),y_j\} - (\gamma_{kj}-\mu_{jk}) \partial_k(a_k)y_j \\
	&= \partial_j(\partial_k(a_k))y_j - (\gamma_{kj}-\mu_{jk}) \partial_k(a_k)y_j \\
	&= \gamma_{kj}\partial_k(a_k)y_j - (\gamma_{kj}-\mu_{jk}) \partial_k(a_k)y_j \\
	&= \mu_{jk}\partial_k(a_k)y_j 
	= \psi_k(-\mu_{jk}x_ky_j) = \psi_k( \{x_k,y_j\} ).
\end{align*}
This proves the claim.
\end{proof}

Set $T=T_0[y_1;\phi_1,\psi_1]_P\cdots[y_n;\phi_n,\psi_n]_P$. As an associative algebra, $T=R[x_1,y_1,\hdots,x_n,y_n]$. The Poisson structure defined above may be summarized as
\begin{align*}
&\{r,x_i\} = \partial_i(r)x_i, \quad 
	\{ r,y_i \} = -\partial_i(r)y_i,  \quad
	\{y_i,x_i\} = \partial_i(a_i),
	& &i \in [n], r \in R, \\
&\{y_i,x_j\} = \mu_{ij}x_jy_i, \quad 
\{x_i,x_j\} = (\gamma_{ij}-\mu_{ij}) x_i x_j, \quad
\{y_i,y_j\} = (\mu_{ji}-\gamma_{ij}) y_i y_j,  & &i,j \in [n], i \neq j.
\end{align*}

\begin{theorem}\label{thm.construct}
Keep the above data. For $i \in [n]$, the element $y_ix_i-a_i$ is Poisson normal in $T$.
Consequently, $T/(y_ix_i-a_i : i \in [n])$ is a Poisson algebra and its bracket is the 
one in Definition \ref{defn.tgwpa}. That is, $A_\mu(R,\bpartial,\ba)$ exists.
\end{theorem}
\begin{proof}
It suffices to prove that the elements $y_ix_i-a_i$ are Poisson normal in $T$.
Fix $i,j \in [n]$ with $i \neq j$ and $r \in R$. Then
\begin{align*}
\{ y_ix_i - a_i, r\}
	&= \{y_ix_i,r\}
	= \{y_i,r\}x_i + \{x_i,r\}y_i
	= (\partial_i(r) - \partial_i(r) ) x_iy_i = 0\\
\{ y_ix_i - a_i, x_j\}
	&= \{y_i,x_j\}x_i + \{ x_i,x_j \}y_i - \{a_i,x_j\} \\
	&= \mu_{ij}x_ix_jy_i + (\gamma_{ij}-\mu_{ij})x_ix_jy_i - \partial_j(a_i)x_j
	= \gamma_{ij}(x_iy_i - a_i) x_j \\
\{ y_ix_i - a_i, y_j \} 
	&= \{y_i,y_j\}x_i + \{ x_i,y_j \}y_i - \{a_i,y_j\} \\
	&= (\mu_{ji}-\gamma_{ij})x_iy_iy_j - \mu_{ji} x_iy_iy_j + \partial_j(a_i)y_j
	= -\gamma_{ij}(y_ix_i-a_i)y_j \\
\{ y_ix_i - a_i, x_i\}
	&= \{y_i,x_i\}x_i + \{ x_i,x_i \}y_i - \{a_i,x_i\}
	= \partial_i(a_i)x_i - \partial_i(a_i)x_i = 0 \\
\{ y_ix_i - a_i, y_i\}
	&= \{y_i,y_i \}x_i + \{ x_i,y_i \}y_i - \{a_i,y_i\} 
	= -\partial_i(a_i)y_i + \partial_i(a_i)y_i = 0.	
\end{align*}
This proves the claim.
\end{proof}

\subsection{Connection with Rees rings}
\label{sec.rees}

Here we present an alternate perspective on TGWPAs that aligns with the theory of Bell--Rogalski (BR) algebras \cite{BR,GRW1,GRW2}. In particular, Theorem \ref{thm.rees} should be compared to \cite[Theorem 2.8]{GRW2}.

Let $R$ be a Poisson algebra and let $a \in \pcnt(R)$ be regular. Set $I=(a)$. Then $R[x,y]/(xy-a) \cong R[t,I t\inv]$ via the map $r \mapsto r$, $x \mapsto t$, and $y \mapsto at\inv$.

Now let $\partial$ be a Poisson derivation of the Poisson algebra $R$ and consider the skew Poisson Laurent algebra $S=R[t^{\pm 1};\partial]_P$. Formally, this is the localization of the Poisson Ore extension $R[t;\partial]_P$ at powers of $t$ and where $\{r,t\inv\}=-\partial(r)t\inv$ for all $r \in R$ by the quotient rule. The subalgebra of $S$ generated by $R$, $t$, and $at\inv$ is isomorphic to the rank one GWPA corresponding to the data $(R,\partial,a)$. 

We extend this idea to TGWPAs. 
Consider the rank $n$ skew-symmetric Poisson algebra $S=R_\bp[\bt^{\pm 1};\bpartial]_P$ where $\bp$ is an $n \times n$ skew-symmetric matrix over $\kk$. The Poisson bracket is given by 
\[
\{r,t_i^{\pm 1}\} = \pm \partial_i(r) t_i^{\pm 1}, \qquad
\{t_i,t_j^{\pm 1}\}=\pm p_{ij} t_it_j^{\pm 1} \qquad
\text{for all $i,j \in [n]$, $i \neq j$, $r \in R$.}
\]
Let $\ba$ be an $n$-tuple of Poisson central regular elements of $R$ and $\bpartial$ be an $n$-tuple of commuting Poisson derivations of $R$ such that $\partial_j(a_i) = \gamma_{ij} a_i$ for $i \neq j$. Let $S'$ be the subalgebra of $S$ generated over $R$ by $t_i$ and $a_it_i\inv$ for all $i \in [n]$. 

\begin{theorem}\label{thm.rees}
Let $A=A_\mu(R,\bpartial,\ba)$ be a TGWPA. Keep the above notation and set $p_{ij}=\gamma_{ij}-\mu_{ij}$ for all $i,j \in [n]$, $i\neq j$. Then $A \iso S'$.
\end{theorem}
\begin{proof}
Note that $\bp=(p_{ij})$ is skew-symmetric by the consistency equations \eqref{eq.gamma}.
Set $A=A_\mu(R,\bpartial,\ba)$ and define a map $\phi:A \to S'$ by $x_i \mapsto t_i$ and $y_i \mapsto a_it_i\inv$ for $i \in [n]$, and $\phi(r)=r$ for $r\in R$. It is well known that $\phi$ is an algebra isomorphism. We claim that $\phi$ extends to a Poisson isomorphism. 

For $i,j \in [n]$, $i \neq j$, and $r,s \in R$, we have
\begin{align*}
\{ \phi(r), \phi(s) \} - \phi( \{r,s\} ) &= 0 \\
\{ \phi(r), \phi(x_i) \} - \phi(\partial_i(r)) \phi(x_i) &= \{ r, t_i\} - \partial_i(r) t_i = 0 \\
\{ \phi(r), \phi(y_i) \} + \phi(\partial_i(r)) \phi(y_i)
	&= \{ r, a_it_i\inv\} + \partial_i(r) a_it_i\inv 
	= a_i (\{r,t_i\inv\} + \partial_i(r) t_i\inv) = 0 \\
\{ \phi(y_i), \phi(x_i) \} - \phi(\partial_i(a_i))
	&= \{ a_it_i\inv,t_i\} - \partial_i(a_i)
	= \{ a_i,t_i\}t_i\inv - \partial_i(a_i)
	=  0 \\
\{ \phi(y_i), \phi(x_j) \} - \mu_{ij} \phi( x_j) \phi( y_i)
	&= \{ a_it_i\inv, t_j\} - \mu_{ij} t_j (a_i t_i\inv) \\
	&= \gamma_{ij}a_i t_i\inv t_j - a_i (p_{ij}t_i\inv t_j) - \mu_{ij} t_j (a_i t_i\inv) \\
	&= a_i(\gamma_{ij} - p_{ij} - \mu_{ij})t_i\inv t_j = 0 \\
\{ \phi(x_i), \phi(x_j) \} - (\gamma_{ij}-\mu_{ij}) \phi( x_i)\phi(  x_j)
	&= \{ t_i, t_j\} - (\gamma_{ij}-\mu_{ij}) t_it_j
	= (p_{ij}-(\gamma_{ij}-\mu_{ij}) )t_it_j = 0 \\
\{ \phi(y_i), \phi(y_j) \} - (\mu_{ji}-\gamma_{ij}) \phi(y_i)\phi(y_j)
	&= \{ a_it_i\inv, a_jt_j\inv \} - (\mu_{ji}-\gamma_{ij}) a_ia_j t_i\inv t_j\inv \\
	&= \left( (-\gamma_{ij} + \gamma_{ji}) t_i\inv t_j\inv + \{t_i\inv,t_j\inv\}
		 - (\mu_{ji}-\gamma_{ij}) t_i\inv t_j\inv \right) a_ia_j\\
	&= \left( \gamma_{ji} - \mu_{ji} - p_{ji}\right) t_i\inv t_j\inv a_i a_j = 0.
\end{align*}
This proves the result.
\end{proof}

\section{Constructions and examples}
\label{sec.construct}

In this section, we show that the TGWPA structure is preserved under several common constructions including tensor products, Poisson twists, and invariant rings. We also present, as an example, multiparameter quantized Weyl Poisson algebras, which were previously studied by Oh \cite{OH_mppwa}.

\subsection{Tensor products}
Let $(R,\{,\}_R)$ and $(S,\{,\}_S)$ be Poisson algebras. Then $R \otimes S$ is a Poisson algebra where
\[ \{ r_1 \otimes s_1, r_2 \otimes s_2 \} = \{ r_1,r_2\}_R \otimes s_1s_2 + r_1r_2 \otimes \{s_1,s_2\}_S\]
for all $r_i \otimes s_i \in R \otimes S$.

By \cite[Theorem 2.16]{GR}, the tensor product of two TGWAs is again a TGWA under certain regularity and consistency conditions. See also \cite[Theorem 3.17]{GR3} for a related result on \emph{twisted} tensor products. The corresponding result for TGWPAs is still true, but much easier to establish and we omit the proof. The direct sum $\mu \oplus \mu'$ below is the usual direct sum of matrices.

\begin{proposition}\label{prop.tensor}
Let $A=A_\mu(R,\bpartial,\ba)$ and $A'=A_{\mu'}(R',\bpartial',\ba')$ be TGWPAs of rank $n$ and $m$, respectively. Set $S = R \otimes R'$ and $\nu=\mu \oplus \mu'$. Define $\bnabla = (\nabla_1,\hdots,\nabla_{n+m})$ and $\bb = (b_1,\hdots,b_{n+m})$ by
\[
\nabla_i = \begin{cases}
	\partial_i \otimes 1 & \text{if $1 \leq i \leq n$} \\
	1 \otimes \partial_{i-n}' & \text{if $n+1 \leq i \leq n+m$},
\end{cases} \qquad
b_i = \begin{cases}
	a_i \otimes 1 & \text{if $1 \leq i \leq n$} \\
	1 \otimes a_{i-n}' & \text{if $n+1 \leq i \leq n+m$}.
\end{cases}
\] 
Then $A \otimes A' \cong A_\nu(S, \bnabla, \bb)$.
\end{proposition}

\subsection{Poisson twists}
\label{sec.twist}

First, we recall some background from \cite{TWZ}.

Let $G$ be an abelian group (herein we will typically use $G=\ZZ^n$), let $(A, \{,\})$ be a $G$-graded Poisson algebra, and let $\delta=\{ \delta_g \mid g \in G \}$ be a set of graded Poisson derivations of $A$. Then $\delta$ is a \emph{Poisson twisting system} if $\delta_g\delta_h = \delta_h\delta_g$ for all $g,h \in G$ and $\delta_{|ab|} = \delta_{|a|}+\delta_{|b|}$ for all homogeneous elements $a,b \in A$. Here $|a|=\deg(a)$. We note that this definition omits condition (3) from \cite[Definition 2.1]{TWZ} because we have assumed that the $\delta_g$ are Poisson derivations. Given these data, one defines a new Poisson bracket on $A$ by
\[ \langle a,b \rangle = \{a,b\} + a\delta_{|a|}(b) - b\delta_{|b|}(a)\]
for all homogeneous elements $a,b \in A$. Then $(A,\langle, \rangle)$ is a Poisson algebra by \cite[Theorem 2.4]{TWZ}, which we denote by $A^\delta$.


\begin{lemma}\label{lem.twist}
Let $A=A_\mu(R,\bpartial,\ba)$ be a TGWPA of rank $n$ and fix $k \in [n]$. Suppose that for a particular choice of $q_{kj},p_{kj} \in \kk$ for all $j \in [n]$, there exists a derivation $\delta_{\be_k}$ of $R[x_1,y_1,\hdots,x_n,y_n]$ satisfying
\begin{enumerate}
\item \label{twist1} $\delta_{\be_k}\restrict{R}$ is a Poisson derivation of $R$
satisfying $\partial_j \delta_{\be_k}\restrict{R} = \delta_{\be_k}\restrict{R}\partial_j$,
\item \label{twist2} $\delta_{\be_k}(a_j) = (q_{kj} + p_{kj})a_j$, and
\item \label{twist3} $\delta_{\be_k}(x_j) = q_{kj} x_j$ and $\delta_{\be_k}(y_j) = p_{kj} y_j$
\end{enumerate}
for all $j \in [n]$. Then $\delta_{\be_k}$ descends to a Poisson derivation of $A$.
\end{lemma}
\begin{proof}
First we show that $\delta_{\be_k}$ preserves $y_jx_j-a_j$ for all $j \in [n]$. We have
\[
\delta_{\be_k}(y_j)x_j + y_j\delta_{\be_k}(x_j) - \delta_{\be_k}(a_j)
	= (p_{kj} + q_{kj})(y_jx_j-a_j).
\]
Hence, $\delta_{\be_k}$ defines an algebra derivation of $A$. We now show that $\delta_{\be_k}$ preserves the Poisson bracket, which it does on $R$ by hypothesis. Fix $r \in R$ and $i,j \in [n]$ with $i \neq j$. Then
\begin{align*}
\{ \delta_{\be_k}(r), x_j \} + \{r, \delta_{\be_k}(x_j) \}
	&= \partial_j(\delta_{\be_k}(r))x_j + q_{kj} \partial_j(r)x_j
	= \delta_{\be_k}( \partial_j(r) ) x_j + \partial_j(r) \delta_{\be_k}(x_j)
	= \delta_{\be_k}( \partial_j(r)x_j) \\
\{ \delta_{\be_k}(r), y_j \} + \{r, \delta_{\be_k}(y_j) \}
	&= -\partial_j(\delta_{\be_k}(r))y_j - p_{kj} \partial_j(r)y_j
	= -\delta_{\be_k}( \partial_j(r) ) y_j - \partial_j(r) \delta_{\be_k}(y_j)
	= \delta_{\be_k}( -\partial_j(r)y_j) \\
\{ \delta_{\be_k}(y_j), x_j\} + \{ y_j, \delta_{\be_k}(x_j) \}
	&= (p_{kj} + q_{kj}) \partial_j(a_j) 
	= \partial_j( (p_{kj} + q_{kj})a_j)
	= \partial_j (\delta_{\be_k}(a_j))
	= \delta_{\be_k}(\partial_j(a_j)) \\
\{\delta_{\be_k}(y_i),x_j\} + \{y_i,\delta_{\be_k}(x_j)\}
	&= \{ p_{ki}y_i,x_j\} + \{y_i, q_{kj} x_j\}
	= \mu_{ij} (p_{ki}y_ix_j + q_{kj}y_i x_j) \\
	&= \mu_{ij} (\delta_{\be_k}(y_i)x_j + y_i\delta_{\be_k}(x_j)) 
	= \delta_{\be_k}(\mu_{ij}y_ix_j) \\
\{\delta_{\be_k}(x_i),x_j\} + \{x_i,\delta_{\be_k}(x_j)\}
	&= \{ q_{ki}x_i,x_j \} + \{ x_i, q_{kj}x_j\}
	= (\gamma_{ij} - \mu_{ij}) (q_{ki}x_ix_j + q_{kj}x_ix_j) \\
	&= (\gamma_{ij} - \mu_{ij}) (\delta_{\be_k}(x_i)x_j + x_i\delta_{\be_k}(x_j)) 
	= \delta_{\be_k}( (\gamma_{ij} - \mu_{ij})x_ix_j) \\
\{\delta_{\be_k}(y_i),y_j\} + \{y_i,\delta_{\be_k}(y_j)\}
	&= \{ p_{ki}y_i,y_j \} + \{ y_i, p_{kj}y_j\}
	= (\mu_{ji}-\gamma_{ij}) (p_{ki}y_iy_j + p_{kj}y_iy_j) \\
	&= (\mu_{ji}-\gamma_{ij})(\delta_{\be_k}(y_i)y_j + y_i\delta_{\be_k}(y_j))
	= \delta_{\be_k}( (\mu_{ji}-\gamma_{ij})y_iy_j).
\end{align*}
This proves the result.
\end{proof}

For arbitrary choices of $q_{kj},p_{kj} \in \kk^\times$, there need not be a derivation satisfying the conditions in Lemma \ref{lem.twist}. In particular, the possible choices of scalars depends heavily on the choice of elements $a_j \in R$. However, if we take $\delta_{\be_k}\restrict{R}=0$ and $q_{kj}=-p_{kj}$ for all $j \in [n]$, then it is easy to see that $\delta_{\be_k}$ satisfies the given conditions.

Let $\delta=\{ \delta_{\be_k} \}_{k \in [n]}$ be a set of pairwise commuting Poisson derivations satisfying the conditions of Lemma \ref{lem.twist}. For all $k \in [n]$, we define $\delta_{-\be_k} := -\delta_{\be_k}$. For $\balpha \in \ZZ^n$, set 
\[ \delta_{\balpha} = \sum_{i=1}^n \alpha_i \delta_{\be_i}.\]
Hence, $\delta_{\bzero}=0$ and for $\balpha,\bbeta \in \ZZ^n$, $\delta_{\balpha + \bbeta} = \delta_{\balpha}+\delta_{\bbeta}$. 
Thus, $\delta = \{ \delta_{\balpha} \mid \balpha \in \ZZ^n\}$ is a Poisson twisting system.

\begin{theorem}\label{thm.twist}
Let $A$ and $\delta$ be defined as above.
Set $\eta_{ij}=\gamma_{ij}-q_{ji}-p_{ji}$ and $\nu_{ij} = \mu_{ij} - q_{ij} - p_{ji}$. 
Let $\bnabla=(\nabla_1,\hdots,\nabla_n)$ where $\nabla_i = \partial_i - \delta_{\be_i}\restrict{R}$ for $i \in [n]$.
Then $A^\delta \iso A_\nu(R,\bnabla,\ba)$ as Poisson algebras.
\end{theorem}
\begin{proof}
By Lemma \ref{lem.twist}, the $\nabla_i$ are commuting Poisson derivations and
\[
\nabla_j(a_i) = \partial_j(a_i) - \delta_{\be_j}\restrict{R}(a_i)
	= \left(\gamma_{ij} - (q_{ji} + p_{ji}) \right)a_i
	= \eta_{ij} a_i.
\]
That is, the tuple $(R,\bnabla,\ba)$ defines a TGWPA. 

By definition, $A^\delta = A_\nu(R,\bnabla,\ba)$ as algebras. We need only show that the bracket $\langle,\rangle$ on $A^{\delta}$ aligns with that on $A_\nu(R,\bnabla,\ba)$. Let $r,s \in R$ and $i,j \in [n]$ with $i \neq j$. We have
\begin{align*}
\langle r,s \rangle
	&= \{r,s\} + r\delta_{\bzero}(s) - s\delta_{\bzero}(r)
	= \{r,s\} \\
\langle r,x_i \rangle 
	&= \{r,x_i\} + r\delta_{\bzero}(x_i) - x_i\delta_{\be_i}(r)
	= \left( \partial_i(r) - \delta_{\be_i}(r) \right) x_i 
	= \nabla_i(r) x_i \\
\langle r,y_i \rangle 
	&= \{r,y_i\} + r\delta_{\bzero}(y_i) - y_i\delta_{-\be_i}(r)
	= \left( -\partial_i(r) + \delta_{\be_i}(r) \right) y_i 
	= -\nabla_i(r) y_i \\
\langle y_i,x_i \rangle 
	&= \{y_i,x_i\} + y_i\delta_{-\be_i}(x_i) - x_i\delta_{\be_i}(y_i) 
	= \partial_i(a_i) - (q_{ii} + p_{ii}) a_i 
	= \nabla_i(a_i) \\
\langle x_i,x_j \rangle &= \{x_i,x_j\} + x_i\delta_{\be_i}(x_j) - x_j\delta_{\be_j}(x_i)
	= \left( (\gamma_{ij}-\mu_{ij}) + q_{ij} - q_{ji} \right) x_i x_j
	= (\eta_{ij}-\nu_{ij})x_ix_j \\
\langle y_i,y_j \rangle &= \{y_i,y_j\} + y_i\delta_{-\be_i}(y_j) - y_j\delta_{-\be_j}(y_i)
	= \left( (\mu_{ji}-\gamma_{ij}) - p_{ij} + p_{ji} \right) y_i y_j
	= (\nu_{ji}-\eta_{ij}) y_iy_j \\
\langle y_i,x_j \rangle &= \{y_i,x_j\} + y_i\delta_{-\be_i}(x_j) - x_j\delta_{\be_j}(y_i)
	= \left( \mu_{ij} - q_{ij} - p_{ji} \right) y_i x_j 
	= \nu_{ij} y_ix_j.
\end{align*}
Finally, we observe that the parameters satisfy the consistency condition \eqref{eq.gamma}:
\begin{align*}
\eta_{ij} + \eta_{ji} 
	&= (\gamma_{ij}-q_{ji}-p_{ji}) + (\gamma_{ji}-q_{ij}-p_{ij}) \\
	&= (\gamma_{ij} + \gamma_{ji}) - (q_{ij} + q_{ji} + p_{ij} + p_{ji}) \\
	&= (\mu_{ij} + \mu_{ji})- (q_{ij} + q_{ji} + p_{ij} + p_{ji}) \\
	&= (\mu_{ij} - q_{ij} - p_{ji}) + (\mu_{ji} - q_{ji} - p_{ij}) \\
	&= \nu_{ij} + \nu_{ji}.
\end{align*}
This proves the result.
\end{proof}

\begin{corollary}\label{cor.twist}
Let $A=A_\mu(R,\bpartial,\ba)$ and $A'=A_\nu(R,\bpartial,\ba)$ be TGWPAs.
Then $A' = A^\delta$ for a twisting system $\delta=\{ \delta_{\be_k} \}_{k \in [n]}$ with $\delta_{\be_k}\restrict{R} = 0$ for all $k \in [n]$. In particular, if $\gamma_{ij}=0$ for all $i\neq j$, then $A$ is a twist of a GWPA.
\end{corollary}
\begin{proof}
We construct a twisting system $\delta$ such that $A'=A^\delta$. Set $c_{ij} = \mu_{ij} - \nu_{ij}$. Then
\[
c_{ij} + c_{ji} = (\mu_{ij} - \nu_{ij}) + (\mu_{ji} - \nu_{ji})
	= (\mu_{ij} + \mu_{ji}) - (\nu_{ij} + \nu_{ji})
	= 0.
\]
Define $\bq \in M_n(\kk)$ by $q_{ii}=0$, $q_{ij} = c_{ij}$ for $i<j$, and $q_{ij} = 0$ for $i>j$. Define $\bp \in M_n(\kk)$ by $p_{ij}=-q_{ij}$. Finally, define $\delta_{\be_k}$ as in Lemma \ref{lem.twist} with $\bp,\bq$ and $\delta_{\be_k}\restrict{R} = 0$. Theorem \ref{thm.twist} now gives $\bnabla = \bpartial$, $\eta = \gamma$, and $\mu_{ij}-q_{ij}+q_{ji}=\nu_{ij}$ for all $i,j$.

For the last claim, assume $\gamma_{ij}=0$ for all $i \neq j$. Then the consistency equation gives $\mu_{ij}+\mu_{ji}=0$, so $\nu=0$ also satisfies \eqref{eq.gamma}. Take $\nu=0$ and let $\delta$ be as above. In $A'$ we have $\eta_{ij}=\gamma_{ij}-q_{ji}-p_{ji} = 0$ for all $i\neq j$, so $A'$ is a GWPA. Since $A=(A^\delta)^{-\delta}$, then $A$ is a twist of a GWPA \cite[Proposition 2.9]{TWZ}.
\end{proof}

\subsection{Invariants}
\label{sec.invariants}

In this section we study invariants of TGWPAs under diagonal automorphisms. We prove a version of a theorem of Jordan and Wells in the case of rank one GWAs \cite[Theorem 2.6]{JW}. The theorem below is a Poisson analogue of \cite[Theorem 3.10]{GR} in the case of TGWAs of type $(A_1)^n$. The question of invariants for various generalizations of Weyl algebras has been considered in a variety of works \cite{GHo,GRW1,GRW2,GW1}. 

\begin{lemma}
Let $A=A_\mu(R,\bpartial,\ba)$ be a TGWPA of rank $n$.
Let $\psi$ be a Poisson automorphism of $R$ such that for all $i \in [n]$,
$\psi \partial_i = \partial_i\psi$ and $\psi(a_i)=a_i$.
Let $\alpha_i \in \kk^\times$ for $i \in [n]$ and define a map $\phi:A \to A$ by
\[ \phi\restrict{R} = \psi, \quad \phi(x_i)=\alpha_i x_i, \quad \phi(y_i) = \alpha_i\inv y_i.\]
Then $\phi$ is a $\ZZ^n$-graded Poisson automorphism of $A$.
\end{lemma}
\begin{proof}
It is easy to see that $\phi$ is an algebra automorphism of $A$. We show that $\phi$ respects the bracket on $A$. By hypothesis, $\phi\restrict{R}$ is a Poisson automorphism. Let $r \in R$ and $i \in [n]$. Then
\begin{align*}
\{ \phi(r), \phi(x_i)\} 
	&= \{ \psi(r), \alpha_i x_i\}
	= \partial_i(\psi(r)) (\alpha_i x_i)
	= \phi( \partial_i(r)x_i ) \\
\{ \phi(y_i), \phi(r) \}
	&= \{ \alpha_i\inv y_i, \psi(r) \}
	= \partial_i(\psi(r)) (\alpha_i\inv y_i)
	= \phi( \partial_i(r)y_i) \\
\{ \phi(y_i),\phi(x_i)\} 
	&= \{ y_i, x_i\}
	= \partial_i(a_i)
	= \partial_i(\psi(a_i))
	= \phi(\partial_i(a_i)).
\end{align*}
The remaining relations involving only the $x_i$ and $y_i$ are easy to check and we omit them.
\end{proof}

Let $A=A_\mu(R,\bpartial,\ba)$ be a TGWPA of rank $n$ and let $\phi$ be an automorphism of $A$ satisfying the following:
\begin{enumerate}
\item $\phi\restrict{R}$ is an automorphism of $R$ with $\ell=\ord\left(\phi\restrict{R}\right) < \infty$,
\item for each $i \in [n]$, $\phi(x_i)=\alpha_i x_i$ and $\phi(y_i)=\alpha_i\inv y_i$ for some $\alpha_i \in \kk^\times$ with $m_i=\ord(\alpha_i)<\infty$, and
\item the integers $\ell,m_1,\hdots,m_n$ are pairwise relatively prime.
\end{enumerate}
Since the underlying commutative ring is a TGWA in which all automorphisms are the identity, then the invariant ring $A^{\grp{\phi}}$ is generated by $R^{\grp{\phi\restrict{R}}}$, $x_i^{m_i}$, and $y_i^{m_i}$ for all $i \in [n]$ \cite[Theorem 3.10]{GR}.

\begin{theorem}\label{thm.invariants}
Let $A$ be a TGWPA of rank $n$. Let $\phi$ be a Poisson automorphism of $A$ satisfying (1), (2), and (3) above. Define $\btheta$ and $\bb$ by $\theta_i=m_i \partial_i$ and $b_i=a_i^{m_i}$ for each $i \in [n]$. 
Define $\nu_{ij} = m_i m_j \mu_{ij}$ and $\eta_{ij} = m_i m_j \gamma_{ij}$ for $i,j \in [n]$ with $i \neq j$. Then $B=A^{\grp{\phi}}$ is a TGWPA and $B \iso A_\nu(R^{\grp{\phi\restrict{R}}},\btheta,\bb)$.
\end{theorem}
\begin{proof}
As an algebra, the invariant ring is
\[ A^{\grp{\phi}} = R^{\grp{\phi\restrict{R}}}[x_1^{m_1},y_1^{m_1},\hdots,x_n^{m_n},y_n^{m_n}]/(x_1^{m_1}y_1^{m_1}-a_1^{m_1},\hdots,x_n^{m_n}y_n^{m_n}-a_n^{m_n}).\]
For $r \in R$ and $i,j \in [n]$ with $i \neq j$ we have 
\begin{align*}
\{ y_i^{m_i}, r \} &= m_i \{y_i,r\} y_i^{m_i-1}  = m_i \partial_i(r) y_i^{m_i} = \theta_i(r) y_i^{m_i}  \\
\{ x_i^{m_i}, r \} &= m_i \{x_i,r\} x_i^{m_i-1}  = -m_i \partial_i(r) x_i^{m_i}  = -\theta_i(r) x_i^{m_i} \\ 
\{ y_i^{m_i}, x_j^{m_j} \}
	&= m_i m_j x_j^{m_j-1} y_i^{m_i-1} \{ y_i, x_j \} 
	= m_i m_j \mu_{ij} x_j^{m_j} y_i^{m_i} 
	= \nu_{ij} x_j^{m_j} y_i^{m_i} \\
\{ x_i^{m_i}, x_j^{m_j} \}
	&= m_i m_j x_i^{m_i-1} x_j^{m_j-1} \{ x_i, x_j \} 
	= (\eta_{ij}-\nu_{ij}) x_i^{m_i} x_j^{m_j} \\
\{ y_i^{m_i}, y_j^{m_j} \}
	&= m_i m_j y_i^{m_i-1} y_j^{m_j-1} \{ y_i, y_j \} 
	= (\nu_{ji}-\eta_{ij}) y_i^{m_i} y_j^{m_j} \\
\{ y_i^{m_i}, x_i^{m_i} \} &= m_i^2 x_i^{m_i-1} y_i^{m_i-1} \{ y_i,x_i \} = m_i^2 a_i^{m_i-1} \partial_i(a_i) = m_i \partial_i(a_i^{m_i}) = \theta_i(a_i^{m_i}).
\end{align*}
Note that
\[ \nu_{ij} + \nu_{ji} = m_im_j(\mu_{ij} + \mu_{ji}) = m_im_j(\gamma_{ij} + \gamma_{ji})
	= \eta_{ij} + \eta_{ji}\]
and so $B$ satisfies the consistency equation \eqref{eq.gamma}.
The result follows.
\end{proof}

\subsection{Multiparameter quantized Weyl Poisson algebras}

The multiparameter quantized Weyl algebra was defined by Maltsiniotis \cite{malt}. The semi-classical limits of these were studied by Oh \cite{OH_mppwa}, who named them \emph{multiparameter quantized Weyl Poisson algebras}. The algebras presented below are isomorphic to Oh's. See also \cite[Proposition 3.4]{LY}.

Let $\bq = (q_1,\hdots, q_n) \in \kk^n$ and let $\Lambda = (\lambda_{ij}) \in M_n(\kk)$ be skew-symmetric. The \emph{multiparameter quantized Weyl Poisson algebra of degree $n$ over $\kk$}, denoted $A_n^{\bq,\Lambda}(\kk)$ is $\kk[x_1,y_1,\hdots,x_n,y_n]$ with Poisson bracket
\begin{align*}
\{ y_i, y_j \} &= \lambda_{ij} y_iy_j \quad \text{for all $i,j$} &
\{ x_i, x_j \} &= (\lambda_{ij}-q_i) x_ix_j \quad \text{for $i<j$} \\
\{ y_i, x_j \} &=  \lambda_{ji} x_jy_i \quad \text{for $i>j$} &
\{ y_i, x_j \} &= (q_i+\lambda_{ji}) x_jy_i \quad \text{for $i<j$} \\
\{ y_i, x_i \} &= z_i \quad \text{for all $i$},
\end{align*}
where
\[ z_i = 1 + \sum_{k =1}^i q_k x_ky_k \qquad \text{for $i \in [n]$}.\]

Let $R=\kk[s_1,\hdots,s_n]$ be a Poisson algebra with trivial bracket. Set $\bfs = (s_1,\hdots,s_n)$ and define Poisson derivations $\bpartial=(\partial_1,\hdots,\partial_n)$ on $R$ by
\begin{align}\label{eq.mqwpa_partial}
\partial_i(s_j) = \begin{cases}
0 & i>j \\
\displaystyle 1+ \sum_{k=1}^i q_k s_k & j=i \\
q_i s_j & i<j.
\end{cases}
\end{align}

Define $\mu = (\mu_{ij}) \in M_n(\kk)$ by
\[ 
\mu_{ij} = \begin{cases}
	\lambda_{ji} & i > j \\
	q_i + \lambda_{ji} & i < j.
\end{cases}
\]

\begin{proposition}\label{prop.mqwpa}
Let $A=A_\mu(R,\bpartial,\bfs)$ be the TGWPA corresponding to the above data. 
Then $A \iso A_n^{\bq,\Lambda}(\kk)$.
\end{proposition}
\begin{proof}
By \eqref{eq.mqwpa_partial}, $\gamma_{ji}=0$ if $i>j$ and $\gamma_{ji}=q_i$ if $i<j$.
Let $X_i,Y_i$ be the standard generators of $A$ and $x_i,y_i$ those of $A'=A_n^{\bq,\Lambda}(\kk)$. Then the map $\Phi:A \to A'$ given by $\Phi(X_i)=x_i$, $\Phi(Y_i) = y_i$, and $\Phi(s_i)=x_iy_i$ clearly extends to an algebra isomorphism. It remains only to show that this is a Poisson homomorphism. 

Let  $i,j \in [n]$. First, suppose $i<j$. Then
\begin{align*}
\{ \Phi(s_i), \Phi(s_j) \}
	&= \{ x_iy_i, x_jy_j\} \\
	&= \{ x_i, x_j\} y_iy_j + \{x_i,y_j\} x_jy_i + \{ y_i, x_j\} x_iy_j + \{y_i,y_j\} x_ix_j \\
	&= \left( (\lambda_{ij}-q_i) - \lambda_{ij} + (q_i + \lambda_{ji}) + \lambda_{ij}  \right) x_ix_jy_iy_j = 0 \\
\{ \Phi(s_j), \Phi(X_i) \} - \Phi(\partial_i(s_j)) \Phi(X_i)
	&= \{ x_j, x_i\}y_j + \{ y_j,x_i\} x_j - q_i x_ix_jy_j \\
	&= ( q_i - \lambda_{ij} + \lambda_{ij} - q_i)x_ix_jy_j = 0 \\
\{ \Phi(s_j), \Phi(Y_i) \} + \Phi(\partial_i(s_j)) \Phi(Y_i)
	&= \{ x_j, y_i\}y_j + \{ y_j,y_i\} x_j + q_i x_jy_iy_j \\
	&= ( -(q_i+\lambda_{ji}) + \lambda_{ji} + q_i)x_jy_iy_j = 0 \\
\{ \Phi(Y_i), \Phi(X_j) \} - \mu_{ij}\Phi(X_j)\Phi(Y_i)
	&= \{ y_i, x_j \} - (q_i+\lambda_{ji})x_jy_i = 0 \\
\{ \Phi(X_i), \Phi(X_j) \} - (\gamma_{ij}-\mu_{ij}) \Phi(X_i)\Phi(X_j)
	&= \{ x_i, x_j \} + (q_i+\lambda_{ji})x_ix_j = 0 \\
\{ \Phi(Y_i), \Phi(Y_j) \} - (\mu_{ji}-\gamma_{ij})\Phi(Y_i)\Phi(Y_j)
	&= \{ y_i, y_j \} - (\lambda_{ij})y_iy_j = 0.
\end{align*}
If $i>j$, then
\begin{align*}
\{ \Phi(s_j), \Phi(X_i) \} - \Phi(\partial_i(s_j)) \Phi(X_i)
	&= \{ x_j, x_i\}y_j + \{ y_j,x_i\} x_j \\
	&= ( \lambda_{ji}-q_j + q_j+\lambda_{ij})x_ix_jy_j = 0 \\
\{ \Phi(s_j), \Phi(Y_i) \} + \Phi(\partial_i(s_j)) \Phi(Y_i)
	&= \{ x_j, y_i\}y_j + \{ y_j,y_i\} x_j \\
	&= ( -\lambda_{ji} - \lambda_{ij} ) x_jy_iy_j = 0 \\
\{ \Phi(Y_i), \Phi(X_j) \} - \mu_{ij}\Phi(X_j)\Phi(Y_i)
	&= \{ y_i, x_j \} - \lambda_{ji}x_jy_i = 0 \\
\{ \Phi(X_i), \Phi(X_j) \} - (\gamma_{ij}-\mu_{ij}) \Phi(X_i)\Phi(X_j)
	&= \{ x_i, x_j \} - (q_j-\lambda_{ji})x_ix_j = 0 \\
\{ \Phi(Y_i), \Phi(Y_j) \} - (\mu_{ji}-\gamma_{ij})\Phi(Y_i)\Phi(Y_j)
	&= \{ y_i, y_j \} - ((q_j + \lambda_{ij}) - q_j)y_iy_j = 0.
\end{align*}
If $i=j$, then
\begin{align*}
\{ \Phi(s_i), \Phi(X_i) \} - \Phi(\partial_i(s_i)) \Phi(X_i)
	&= \{ x_iy_i, x_i \} - \Phi\left( 1+ \sum_{k=1}^i q_k s_k \right)x_i
	= \left( \{ y_i,x_i\} - z_i \right)x_i = 0
\\
\{ \Phi(s_i), \Phi(Y_i) \} + \Phi(\partial_i(s_i)) \Phi(Y_i)
	&= \{ x_iy_i, y_i \} + \Phi\left( 1+ \sum_{k=1}^i q_k s_k \right)y_i
	= \left( \{ x_i,y_i\} + z_i \right)y_i = 0.
\end{align*}
Finally, for all $i \in [n]$, 
\begin{align*}
\{ \Phi(Y_i), \Phi(X_i) \} - \Phi( \partial_i(s_i) ) 
	&= \{y_i,x_i\} - \Phi\left( 1+ \sum_{k=1}^i q_k s_k \right)
	= z_i - z_i = 0.
\end{align*}
This proves the result.
\end{proof}

\section{Simplicity}
\label{sec.simple}

In this final section, we give a criterion for simplicity of TGWPAs. This criterion is analogous to the one for (T)GWAs and essentially the same as that for GWPAs, as given by Bavula \cite[Theorem 1.1]{Bpoisson}, though our proof is different. Throughout this section, we assume $\mathrm{char}{\kk}=0$.

We first prove a simplicity criterion for the skew Poisson Laurent extensions.
This was proved in the rank one case by Oh \cite[Lemma 3.3]{OH1}, which is a Poisson analogue of a result of Jordan \cite[Theorem 1]{jsimp1}. See also \cite[Proposition 5.4]{bell1} and \cite[Corollary 10]{jespers}.

Let $R$ be a Poisson algebra, $n \in \ZZ_+$, and $\bpartial=(\partial_1,\hdots,\partial_n)$ an $n$-tuple of commuting Poisson derivations. For $\balpha \in \ZZ^n$, write $\partial_{\balpha} = \sum_{i=1}^n \alpha_i \partial_i$. Let $\bp \in M_n(\kk)$ be skew symmetric and set $B=R[x_1^{\pm 1},\hdots,x_n^{\pm 1};\bpartial]_P$ where $\{ x_i,x_j\} = p_{ij}x_ix_j$. In $B$, we use the notation $x^{\balpha}$ to denote the monomial $x_1^{\alpha_1} \cdots x_n^{\alpha_n}$, and for $i \in [n]$ we use $\xi_{\balpha,i}$ to denote the scalar such that $\{x_i,x^{\balpha}\}=\xi_{\balpha,i} x^{\balpha+\be_i}$. Explicitly, $\xi_{\balpha,i} = \sum_j \alpha_j p_{ij}$. Recall that the \emph{support} of $f = \sum_{\balpha} c_{\balpha} x^{\balpha} \in B$ is the set of $\balpha \in \ZZ^n$ such that $c_{\balpha} \neq 0$.


\begin{lemma}\label{lem.laurent}
Keep the above setup. Then $B$ is Poisson simple if and only if 
\begin{enumerate}
\item \label{lnt1} $R$ has no nonzero proper $\bpartial$-invariant Poisson ideals, and
\item \label{lnt2} there is no $\balpha \in \ZZ^n\backslash\{\bzero\}$ and nonzero $u \in R$ such that $u\partial_{\balpha}(r)=\{u,r\}$ for $r \in R$ and $\partial_i(u)=\xi_{\balpha,i} u$ for all $i$.
\end{enumerate}
\end{lemma}
\begin{proof}
($\Rightarrow$) Assume that $B$ is Poisson simple. 

Let $I$ be a nonzero proper $\bpartial$-invariant Poisson ideal of $R$. 
We have $IB = \bigoplus_{\balpha \in \ZZ^n} Ix^{\balpha}$. Then for $i \in [n]$,
\[ \{x_i,IB\} = \bigoplus_{\balpha \in \ZZ^n} \{ x_i, I x^{\balpha}\}
	=  \bigoplus_{\balpha \in \ZZ^n} \left(\partial_i(I) x^{\balpha+\be_i} + I \{x_i,x^{\balpha}\} \right)
	\subset IB.\]
Hence, $IB$ is a Poisson ideal of $B$. 
Thus, $1 \in IB$ so $1 \in I$. A contradiction.

Suppose there exists $\balpha \in \ZZ^n\backslash\{\bzero\}$ and nonzero $u \in R$ such that $u\partial_{\balpha}(r)=\{u,r\}$ for $r \in R$ and $\partial_i(u)=\xi_{\balpha,i} u$ for all $i$.
Let $r \in R$, then
\begin{align*}
 \{ r, ux^{\balpha} \} 
 	&= \{ r,u\}x^{\balpha} +  \partial_{\balpha}(r)ux^{\balpha}
 	= -u\partial_{\balpha}(r) x^{\balpha} +  \partial_{\balpha}(r)ux^{\balpha} = 0 \\
\{ ux^{\balpha}, x_i \} 
	&= x^{\balpha}\{ u, x_i\} + u\{ x^{\balpha}, x_i \}
	= x^{\balpha} \partial_i(u)x_i - u \xi_{\balpha,i} x^{\balpha} x_i
	= 0.
\end{align*}
If $u$ is not a unit, then $ux^{\balpha}$ is Poisson central and so defines a proper, nonzero ideal, contradicting simplicity. If $u$ is a unit, then reducing modulo a maximal ideal in $R$ shows that $ux^{\balpha}-1$ is not a unit. Since $ux^{\balpha}-1$ is Poisson central, then we obtain the same contradiction.

($\Leftarrow$) Assume \eqref{lnt1} and \eqref{lnt2} hold.
Let $I$ be a nonzero Poisson ideal of $B$. Let $f \in I$ be a nonzero element of minimal support $S$. If $|S|=1$, then $f$ is a monomial $rx^{\bbeta}$ so $r=rx^{\bbeta}x^{-\bbeta} \in I \cap R$. Suppose $|S|=k>1$. Taking the standard lexicographic ordering on $\ZZ^n$, let $J \subset R$ be the set of leading coefficients of elements in $I$ with support $S$ along with $0$. Standard arguments show that $J$ is a Poisson ideal of $R$. Let $g \in I$ such that $g=\sum_{\balpha \in S} c_{\balpha} x^{\balpha}$. For $i \in [n]$, we have $\{g,x_i\}x_i\inv \in I$ and
\begin{align}\label{eq.inv_elt}
\{g,x_i\}x_i\inv = \sum_{\balpha \in S} \{ c_{\balpha} x^{\balpha}, x_i\} x_i\inv
	= \sum_{\balpha \in S} \left( \{ c_{\balpha}, x_i\} x^{\balpha} + c_{\balpha} \{ x^{\balpha}, x_i\} \right) x_i\inv
	= \sum_{\balpha \in S} \left( \partial_i(c_{\balpha}) - c_{\balpha} \xi_{\balpha,i} \right) x^{\balpha}.
\end{align}
It follows that $\partial_i(c_{\balpha}) - c_{\balpha} \xi_{\balpha,i} \in J$ but since $c_{\balpha} \in J$, then we must have $\partial_i(c_{\balpha}) \in J$, so $J$ is $\bpartial$-invariant. Thus, $1 \in J$ and so, after multiplying by an appropriate monomial, we may assume that 
\begin{align}\label{eq.newg}
	g = 1 + \sum_{\bzero \neq \balpha \in S} c_{\balpha} x^{\balpha}.
\end{align}
Let $r \in R$, then
\[
\{ r, g\} = \sum_{\bzero\neq\balpha \in S}  \{r, c_{\balpha} x^{\balpha} \}
	= \sum_{\bzero\neq\balpha \in S} \left( \{r, c_{\balpha}\} + c_{\balpha}\partial_{\balpha}(r) \right) x^{\balpha}.
\]
In particular, $\{r,g\}$ has smaller support and so by minimality, $\{r,g\} = 0$. That is, for each $\balpha \in S$ with $\balpha \neq 0$, we have $\{c_{\balpha}, r\} = c_{\balpha}\partial_{\balpha}(r)$.
Taking $\{g,x_i\}x_i\inv$ as in \eqref{eq.inv_elt} but with \eqref{eq.newg} shows that the constant term vanishes, and so by minimality, $\{g,x_i\}=0$. Thus, $\partial_i(c_{\balpha})=c_{\balpha} \xi_{\balpha,i}$. This violates \eqref{lnt2}. 


We conclude that $I \cap R \neq 0$. Let $c \in I \cap R$ with $c \neq 0$. Then $\{c,x_i\}x_i\inv = \partial_i(c) \in I\cap R$. Thus, $I \cap R$ is $\bpartial$-invariant, so $I=B$ by \eqref{lnt1}.
\end{proof}

We now introduce shorthand similar to $x^{\balpha}$ that will be useful in the TGWPA $A=A_\mu(R,\bpartial,\ba)$. For $\balpha \in \ZZ^n$, we let $z^{\balpha}$ denote the monomial $z_1^{\alpha_1} z_2^{\alpha_2} \cdots z_n^{\alpha_n}$ where 
\[
z_i^{\alpha_i} = \begin{cases}
x_i^{\alpha_i} & \text{if $\alpha_i > 0$} \\
y_i^{|\alpha_i|} & \text{if $\alpha_i<0$} \\
1 & \text{if $\alpha_i=0$.}
\end{cases}\]

Before the next lemma, which establishes the necessity of a condition in our simplicity criterion, it will be useful to discuss brackets involving the $z^{\balpha}$.

If $i,j \in [n]$ with $i \neq j$ and $m \geq 1$, then
\begin{align*}
\{ x_i^{m}, x_j \} 
	&= m x_i^{m-1} \{x_i,x_j\} 
	= m (\gamma_{ij}-\mu_{ij}) x_i^{m-1}  x_i x_j 
	= m (\gamma_{ij}-\mu_{ij}) x_i^{m} x_j \\ 
\{ y_i^{m}, x_j\} 
	&= m y_i^{m-1} \{y_i,x_j\} 
	= m y_i^{m-1} \mu_{ij} y_ix_j
	= m \mu_{ij} y_i^{m} x_j.
\end{align*}
Let $\balpha \in \ZZ^n$. We extend our notation from earlier. If $\alpha_j = 0$, then there exists a scalar $\xi_{\balpha,j}$ such that $\{x_j, z^{\balpha}\} = \xi_{\balpha,j} z^{\balpha}x_j = \xi_{\balpha,j} z^{\balpha+\be_j}$. By similar reasoning, there exists a scalar $\xi_{\balpha,j}'$ such that $\{y_j, z^{\balpha}\} = \xi_{\balpha,j}' z^{\balpha}y_j = \xi_{\balpha,j}' z^{\balpha-\be_j}$.


\begin{lemma}\label{lem.gcd}
Let $A=A_\mu(R,\bpartial,\ba)$ be a TGWPA. Suppose that $I=Ra_i + R\partial_i(a_i) \neq R$ for some $i \in [n]$. Then 
\[ 
J = \bigoplus_{\balpha \in \ZZ^n, \alpha_i \neq 0} Rz^{\balpha} \oplus 
	\bigoplus_{\balpha \in \ZZ^n, \alpha_i = 0} Iz^{\balpha}
\]
is a proper, nonzero Poisson ideal of $A$.
\end{lemma}
\begin{proof}
First note that $I$ is a Poisson ideal of $R$ because $a_i$ and $\partial_i(a_i)$ are Poisson central (see Lemma \ref{lem.cnt}). Moreover, if $j \in [n]$ with $i \neq j$, then for $r,s \in R$, we have
\begin{align*}
\partial_j(ra_i + s\partial_i(a_i))
	&= \partial_j(r)a_i + r\partial_j(a_i) + \partial_j(s)\partial_i(a_i) + s\partial_j(\partial_i(a_i)) \\
	&= (\partial_j(r) + \gamma_{ij}r) a_i + (\partial_j(s) + s \gamma_{ij})\partial_i(a_i) \in I.
\end{align*}
Thus, $\partial_j(I) \subset I$ when $i \neq j$.

It is clear that $J$ is a proper, nonzero ideal of $A$. We need only show that $J$ is in fact a Poisson ideal. 
Since $\{ r, z^{\balpha}\} = \partial_{\balpha}(r) z^{\balpha}$ for any $\balpha \in \ZZ^n$ and $r \in R$, then it follows that $\{R,J\} \subset J$.

Suppose $\balpha \in \ZZ^n$ with $\alpha_i \neq 0$. 
Let $j \in [n]$ with $j \neq i$. If $\alpha_j \geq 0$, then 
\[
\{x_j, z^{\balpha} \} 
	= z_j^{\alpha_j} \{ x_j, z^{\balpha-\alpha_j\be_j} \}
	= z_j^{\alpha_j} \xi_{\balpha-\alpha_j\be_j,j} z^{\balpha-\alpha_j\be_j+\be_j}
	= \xi_{\balpha-\alpha_j\be_j,j} z^{\balpha+\be_j}.\]
If $\alpha_j < 0$, then
\begin{align*}
\{x_j, z^{\balpha} \} 
	&= z_j^{\alpha_j} \{ x_j, z^{\balpha-\alpha_j\be_j} \} + z^{\balpha-\alpha_j\be_j} \{ x_j, y_j^{|\alpha_j|} \}  \\
	&= a_j \xi_{\balpha-\alpha_j\be_j,j} z^{\balpha+\be_j} - \alpha_j z^{\balpha-\alpha_j\be_j} y_j^{|\alpha_j|-1} \{ x_j, y_j \} \\
	&= a_j \xi_{\balpha-\alpha_j\be_j,j} z^{\balpha+\be_j} + \alpha_j z^{\balpha-\alpha_j\be_j} z_j^{\alpha_j+1} \partial_j(a_j) \\
	&= \left(a_j \xi_{\balpha-\alpha_j\be_j,j}  + \alpha_j \partial_j(a_j) \right)z^{\balpha+\be_j}.
\end{align*}
Now in any case, for $r \in R$, we have
\begin{align*}
\{ rz^{\balpha}, x_j\} = \{r,x_j\}z^{\balpha} + r\{z^{\balpha},x_j\}
	\in Rz^{\balpha+\be_j}.
\end{align*}
Note that $(\balpha+\be_j)_i \neq 0$ so that $\{rz^{\balpha}, x_j\} \in J$. A similar argument shows that when $\alpha_i \geq 1$ or $\alpha_i \leq -2$, then $\{z^{\balpha}, x_i \} \in J$. If $\alpha_i = -1$, then
\[
	\{z^{\balpha}, x_i \} = z^{\balpha+\be_i} \{y_i,x_i\} + y_i\{ z^{\balpha+\be_i}, x_i\}
		= (\partial_i(a_i) - a_i \xi_{\balpha+\be_i,i} )z^{\balpha+\be_i}.
\]
Since $(\balpha+\be_i)_i = 0$ and $(\partial_i(a_i) + a_i \xi_{\balpha+\be_i,i} ) \in I$, then $\{z^{\balpha}, x_i \} \in J$.
Similarly, if $\alpha_i = 0$, then for $b \in I$
\begin{align*}
\{ bz^{\balpha}, x_j\} 
	&= \{b,x_j\} z^{\balpha} + b \{z^{\balpha}, x_j \}
	\in Rz^{\balpha+\be_j}
\end{align*}
If $i=j$, then $(\balpha+\be_i)_i \neq 0$, so it is clear that $\{ bz^{\balpha}, x_j\} \in J$.
If $i \neq j$, then $(\balpha+\be_j)_i = 0$ and the $\partial_j$-invariance of $I$ implies that $\{ bz^{\balpha}, x_j\} \in J$.

Furthermore, one can repeat the above computations with $y_j$ instead of $x_j$ for analogous results.
Hence, $J$ is a Poisson ideal of $A$.
\end{proof}


We are now ready to prove our main result on simplicity of TGWPAs.

\begin{theorem}\label{thm.simple}
Suppose $\mathrm{char}{\kk}=0$.
Let $A=A_\mu(R,\bpartial,\ba)$ be a TGWPA of rank $n$ over a Poisson algebra $R$. Then $A$ is Poisson simple if and only if the following hold:
\begin{enumerate}
\item \label{sim1} $R$ has no nonzero proper $\bpartial$-invariant Poisson ideals,
\item \label{sim2} there is no $\balpha \in \ZZ^n\backslash\{\bzero\}$ and nonzero $u \in R$ such that $u\partial_{\balpha}(r)=\{u,r\}$ for $r \in R$ and $\partial_i(u)=\xi_{\balpha,i} u$ for all $i$, and
\item \label{sim3} $Ra_i + R\partial_i(a_i)=R$ for all $i=1,\hdots,n$.
\end{enumerate}
\end{theorem}
\begin{proof}
($\Rightarrow$) Suppose $A$ is simple. 
Then \eqref{sim3} follows from Lemma \ref{lem.gcd}.
Let $\cX$ be the multiplicative set generated by 1 and the $x_i$, and set $B=A\cX\inv \iso R_\bp[x_1^{\pm 1},\hdots,x_n^{\pm 1};\bpartial]_P$ (Theorem \ref{thm.rees}), where by an abuse of notation we use $\partial_i$ to denote the extension of the $\partial_i$ to the localization.
By regularity of the $a_i$ (and hence the $x_i$), $A \subset B$ and any nonzero ideal of $B$ meets $A$ nontrivially. Hence, $B$ is simple and so \eqref{sim1} and \eqref{sim2} follow by Lemma \ref{lem.laurent}.  

($\Leftarrow$) Suppose \eqref{sim1}, \eqref{sim2}, and \eqref{sim3} hold. Set $B=A\cX\inv$ as above.
Let $I$ be a nonzero Poisson ideal of $A$. 
By \eqref{sim1} and \eqref{sim2}, $B$ is simple by Lemma \ref{lem.laurent}. Hence, $1 \in IB$, so $z^{\balpha} \in I$ for some $\balpha \geq 0$.
Choose $i \in [n]$ such that $\alpha_i \neq 0$. Then $z^{\balpha}y_i = a_i z^{\balpha-\be_i}$ and
\begin{align*} 
\{ y_i, z^{\balpha}\} 
	&= \{ y_i, x_i^{\alpha_i} z^{\balpha-\alpha_i\be_i} \} \\
	&= z^{\balpha-\alpha_i\be_i} \{ y_i, x_i^{\alpha_i} \} + x_i^{\alpha_i} \{ y_i, z^{\balpha-\alpha_i\be_i} \} \\
	&= z^{\balpha-\alpha_i\be_i}(\alpha_i \partial_i(a_i) x_i^{\alpha_i-1}) + x_i^{\alpha_i} (\xi_{\balpha-\alpha_i\be_i, i}' y_i z^{\balpha-\alpha_i\be_i}) \\
	&= \alpha_i\partial_i(a_i)z^{\balpha-\be_i} + \xi_{\balpha-\alpha_i\be_i, i}' y_iz^{\balpha}.
\end{align*}
We have $\{ y_i, z^{\balpha}\} \in I$ and $y_iz^{\balpha} \in I$. As $\mathrm{char}\kk=0$, we conclude that $\partial_i(a_i)z^{\balpha-\be_i} \in I$. By \eqref{sim3}, there exist $r,s \in R$ such that $ra_i+s\partial_i(a_i)=1$. Consequently,
\[ z^{\balpha-\be_i} = r(a_i z^{\balpha-\be_i}) + s(\partial_i(a_i)z^{\balpha-\be_i}).\]
Continuing in this way, we obtain $1 \in I$, so $I=A$.
\end{proof}

\subsection*{Statement on AI use}
This paper was written independently of AI. After receiving a referee report, Fable 5.2 was used for additional editing. This identified some typos and minor mathematical erros, which were then corrected.


\end{document}